\documentclass[10pt,notitlepage]{amsart}
\usepackage[T1]{fontenc}
\usepackage[latin9]{inputenc}
\usepackage[letterpaper]{geometry}
\usepackage{latexsym, amsfonts, amsmath, amssymb, amsthm, cite, enumerate}

\usepackage{graphicx}
\usepackage{mathrsfs}

\newtheorem{theorem}{Theorem}
\newtheorem{lemma}[theorem]{Lemma}

\theoremstyle{remark}

\newtheorem{remark}{Remark}

\newtheorem*{problem*}{Problem}

\newcommand{\bc}{\mathcal{B}}

\newcommand{\sg}{S_{2g}(V=\square)}
\newcommand{\sgm}{S_{2g-1}(V=\square)}

\newcommand{\fqx}{\mathbb{F}_q[x]}

\numberwithin{theorem}{section} \numberwithin{equation}{section}
\begin{document}

\title{The second and third moment of $L(\frac{1}{2},\chi)$ in the hyperelliptic ensemble}
\date{}
\author{Alexandra Florea}
\address{Department of Mathematics, Stanford University, Stanford, CA 94305}
\email{amusat@stanford.edu}
\maketitle

\begin{abstract}
We obtain asymptotic formulas for the second and third moment of quadratic Dirichlet $L$--functions at the critical point, in the function field setting. We fix the ground field $\mathbb{F}_q$, and assume for simplicity that $q$ is a prime with $q \equiv 1 \pmod 4$.  We compute the second and third moment of $L(1/2,\chi_D)$ when $D$ is a monic, square-free polynomial of degree $2g+1$,  as $g \to \infty$. The answer we get for the second moment agrees with Andrade and Keating's conjectured formula in \cite{conjectures}. For the third moment, we check that the leading term agrees with the conjecture. \end{abstract}

\section{Introduction}

In this paper, we study the second and third moment of quadratic Dirichlet $L$--functions in the function field setting. We obtain asymptotic formulas for 
$$ \sum_{D \in \mathcal{H}_{2g+1}} L \big( \tfrac{1}{2}, \chi_D \big)^k,$$ when $k=2,3$ and $g \to \infty$, where $\mathcal{H}_{2g+1}$ denotes the space of monic, square-free polynomials of degree $2g+1$ over $\mathbb{F}_q[x]$. In our calculation, we take $q$ to be a prime with $q \equiv 1 \pmod 4$. 
More precisely, we prove the following.
\begin{theorem}
Let $q$ be a prime with $q \equiv 1 \pmod 4$. Then
$$ \sum_{D \in \mathcal{H}_{2g+1}} L \big( \tfrac{1}{2}, \chi_D \big)^2 = \frac{q^{2g+1}}{\zeta(2)} P(2g+1) + O(q^{g(1+\epsilon)}),$$ where $\mathcal{H}_{2g+1}$ denotes the space of monic, square-free polynomials of degree $2g+1$ over $\mathbb{F}_q[x]$, $\zeta$ is the zeta function associated with $\mathbb{F}_q[x]$ and $P(x)$ is a polynomial of degree $3$ whose coefficients will be computed explicitly. \label{mresult}
\end{theorem}
\begin{theorem}
Under the same assumptions as above, we have that
$$ \sum_{D \in \mathcal{H}_{2g+1}} L \big( \tfrac{1}{2}, \chi_D \big)^3 = \frac{q^{2g+1}}{\zeta(2)} Q(2g+1) + O(q^{3g/2(1+\epsilon)}),$$ where $Q(x)$ is a polynomial of degree $6$ whose coefficients can be computed explicitly.
\label{th2}
\end{theorem}
There has been a long-standing interest in understanding moments of families of $L$--functions. For the zeta-function, if we define
$$M_k(T) = \frac{1}{T} \int_0^T \left| \zeta \left( \frac{1}{2} + it \right) \right| ^{2k} \, dt,$$ the problem is to find asymptotic formulas for $M_k$ as $T \to \infty$. The leading term for the second moment was computed in \cite{hardyl} to be
$$ M_1 \sim \log T,$$ and the fourth moment leading term was computed by Ingham \cite{ingham}
$$M_2(T) \sim \frac{1}{2 \pi^2} \log^4(T).$$ No other higher moments have been computed so far, but it is conjectured that
$$M_k(T) \sim C_k (\log T)^{k^2},$$ for all $k>0$.  A precise value for $C_k$ was conjectured by Keating and Snaith \cite{ksnaith} using random matrix theory.

One can look at other families of $L$--functions. For example, considering the family of Dirichlet $L$-functions $L(s,\chi_d)$, we are interested in
\begin{equation} 
\sum_{0<d \leq D} L \left( \frac{1}{2}, \chi_d \right)^k,
\label{mom}
\end{equation} where the sum is over real primitive Dirichlet characters. It is conjectured that the $k^{\text{th}}$ moment above is asymptotic to $ C_k D (\log D)^{k(k+1)/2}.$ Jutila \cite{jutila} computed the first and second moment and Soundararajan \cite{sound} computed the second moment $\sum L \left( \frac{1}{2}, \chi_{8d} \right)^2$ and the third moment $\sum L \left( \frac{1}{2}, \chi_{8d} \right)^3$, where the sum is over square-free, odd, positive $d$. Keating and Snaith \cite{keatingsnaith} conjectured the leading term for \eqref{mom}, again using random matrix theory. The other principal lower order terms have been conjectured by Conrey, Farmer, Keating, Rubinstein and Snaith in \cite{cfkrs}. 

In \cite{dgh}, Diaconu, Goldfeld and Hoffstein use multiple Dirichlet series to study the moments of $L(1/2,\chi_d)$. Their work suggests the existence of a lower order term of size $X^{3/4}$ for the cubic moment. Zhang \cite{zhang} conjectured a value for the constant associated with this term. Young \cite{young2} considered the smoothed third moment of this family of $L$-functions and bounded the remainder term by $O(X^{3/4+\epsilon})$.

In this paper, we are interested in the analogous problem of moments of $L$-functions over function fields. Andrade and Keating \cite{keatingandrade} computed the mean value of $L(1/2,\chi_D)$ averaged over monic square-free polynomials of degree $2g+1$. When the cardinality of the field $\mathbb{F}_q$ is $q \equiv 1 \pmod 4$, they proved that
\begin{equation}
 \sum_{D \in \mathcal{H}_{2g+1}} L \big( \tfrac{1}{2}, \chi_D \big) = \frac{P(1)}{2 \zeta(2)} q^{2g+1} \left[ (2g+1)+1 + \frac{4} {\log q} \frac{P'}{P}(1) \right] + O(q^{(2g+1)(3/4+\frac{ \log_q 2}{2})}),
 \label{first}
 \end{equation}  where $\mathcal{H}_{2g+1}$ denotes the space of monic, square-free polynomials of degree $2g+1$ over $\mathbb{F}_q[x]$ and
$$P(s) = \prod_{\substack{P \text{ monic} \\ \text{irreducible}}} \left( 1- \frac{1}{(|P|+1)|P|^s} \right).$$ 

Extending the recipe in \cite{cfkrs} to the function field setting, Andrade and Keating \cite{conjectures} conjectured formulas for integral moments of $L$-functions over function fields. More precisely, they conjectured that
\begin{equation}
 \sum_{D \in \mathcal{H}_{2g+1}} L \big( \tfrac{1}{2}, \chi_D \big)^k =  \sum_{D \in \mathcal{H}_{2g+1}} Q_k (2g+1)(1+ o(1)), \label{conjmoment}
\end{equation}
where $Q_k$ is a polynomial of degree $k(k+1)/2$ given by 
$$ Q_k(x) = \frac{ (-1)^{k(k-1)/2} 2^k}{k!} \frac{1}{(2 \pi i)^k} \oint \ldots \oint \frac{G(z_1, \ldots, z_k) \Delta(z_1^2, \ldots, z_k^2)^2}{ \prod_{j=1}^k z_j^{2k-1}} q^{\frac{x}{2} \sum_{j=1}^k z_j } \, dz_1 \ldots dz_k,$$ and
$$G(z_1, \ldots,z_k) = A \left( \frac{1}{2};z_1,\ldots,z_k \right) \prod_{j=1}^k X \left( \frac{1}{2} + z_j \right)^{-1/2} \prod_{1 \leq i \leq j \leq k} \zeta(1+z_i+z_j).$$
In the above,
$$X(s)= q^{-1/2+s},$$ and
\begin{align}
 A \left( \frac{1}{2}; z_1, \ldots, z_k \right) &= \prod_{\substack{P \text{ monic} \\ \text{irreducible}}} \prod_{1 \leq i \leq j \leq k} \left( 1- \frac{1}{|P|^{1+z_i+z_j}} \right) \nonumber \\
 & \times \left( \frac{1}{2} \left( \prod_{j=1}^k \left(1- \frac{1}{|P|^{1/2+z_j}} \right)^{-1} +  \prod_{j=1}^k \left(1+ \frac{1}{|P|^{1/2+z_j}} \right)^{-1} \right) +\frac{1}{|P|} \right) \left( 1+ \frac{1}{|P|} \right)^{-1}. \label{az}
\end{align}
For $k=1$, the conjecture above agrees with the computed first moment \eqref{first}. Recently, Rubinstein and Wu \cite{rubinstein} provided numerical evidence in favor of these conjectures. 

When $k=2$, the conjectured formula \eqref{conjmoment} simplifies to $  \sum_{D \in \mathcal{H}_{2g+1}} L \left( \frac{1}{2}, \chi_D \right)^2 = \frac{q^{2g+1}}{\zeta(2)} R(2g+1) + o(q^{2g+1}), $ where
\begin{align}
R(x) &= \frac{1}{24 \log(q)^3} \bigg[ (6+11x+6x^2+x^3) A(0,0) (\log q)^3 \nonumber \\
&+ (11+12x+3x^2) (\log q)^2 (A_2(0,0)+A_1(0,0)) + 12(2+x) (\log q) A_{12}(0,0)  \nonumber \\
& - 2(A_{222}(0,0)-3A_{122}(0,0)-3A_{112}(0,0)+A_{111}(0,0)) \bigg], \label{conjpol}
\end{align} where the $A_j$ above are partial derivatives of $A (1/2;z_1,z_2)$ evaluated at $z_1=z_2=0$. Our answer in Theorem \ref{mresult} agrees with the conjecture \eqref{conjpol}.

For the third moment, Andrade and Keating \cite{conjectures} conjecture that
\begin{equation} \sum_{D \in \mathcal{H}_{2g+1}} L \big( \tfrac{1}{2} ,\chi_D \big)^3 \sim \frac{1}{2880 \zeta(2)} A_3 \big( \tfrac{1}{2}; 0,0,0 \big) |D| (\log_q|D|)^6, \label{thirdconj} \end{equation} with
\begin{equation}
A_3 \big( \tfrac{1}{2}; 0,0,0 \big) = \prod_P \left( 1 - \frac{12|P|^5-23|P|^4+23|P|^3-15|P|^2+6|P|-1}{|P|^6(|P|+1)} \right). \label{a3} \end{equation}
We obtain an asymptotic formula for the third moment with an error of size $O(q^{3g/2(1+\epsilon)})$ and we check that the leading term agrees with \eqref{thirdconj}. Checking by hand that all the other lower order terms match the conjecture \eqref{conjmoment} involves laborious computations, and we do not carry them out here. 
\section{Background and setup of the problem} 
 We introduce the notation we use throughout the paper. Let $\mathcal{M}$ denote the monic polynomials over $\fqx, \mathcal{M}_n$ the monic polynomials of degree $n$ over $\fqx$ and $\mathcal{M}_{\leq n}$ the monic polynomials of degree less than or equal to $n$. Then $| \mathcal{M}_n | = q^n$ and $|\mathcal{M}_{\leq n} | = 1+q+ \ldots+ q^n= (q^{n+1}-1)/(q-1)$. 

Let $\mathcal{H}_{d,q}$ denote the set of monic square-free polynomials of degree $d$ over $\fqx$. For ease of notation, we will write it as $\mathcal{H}_d$. The norm of a polynomial $f \in \mathbb{F}_q[x]$ is defined as $|f|=q^{d(f)}$, where for simplicity, $d(f) = \deg(f)$. $d_k(f)$ will denote the $k^{\text{th}}$ divisor function (i.e. $d_k(f) = \displaystyle \sum_{f_1 \cdot \ldots \cdot f_k =f} 1$.)  
From now on, $P$ will be used to denote a monic irreducible polynomial. 
\subsection{Basic facts about $L$-functions over function fields}
Many of the facts stated in this section are proven in \cite{rosen}. 

For $\text{Re}(s)>1$, the zeta function of $\mathbb{F}_q[x]$ is defined by
$$\zeta(s) = \sum_{f \in \mathcal{M}} \frac{1}{|f|^s} = \prod_P (1- |P|^{-s})^{-1}.$$ One can show that $\zeta(s) = (1-q^{1-s})^{-1}$. With the change of variables $u=q^{-s}$, we have $\mathcal{Z}(u)=(1-qu)^{-1}$. 

To determine the cardinality of $\mathcal{H}_n$, consider the generating series
$$ \sum_{\substack{D \text{ monic} \\ \text{square-free}}} u^{d(D)} = \frac{\mathcal{Z}(u)}{\mathcal{Z}(u^2)} = \frac{1-qu^2}{1-qu}.$$ Looking at the coefficient of $u^n$, we see that for $n=1, |\mathcal{H}_1| = q$ and for $n \geq 2, |\mathcal{H}_n|=q^n(1-1/q)= q^n/\zeta(2)$.

For a monic irreducible polynomial $P$, define the quadratic residue $\displaystyle \left( \frac{f}{P} \right)$ by
$$\left( \frac{f}{P} \right) = 
\begin{cases}
1 & \mbox{ if } P \nmid f  \text{ and } f \text{ is a square} \pmod P \\
-1 & \mbox{ if } P \nmid f  \text{ and } f \text{ is not a square} \pmod P \\
0 & \mbox{ if } P|f .
\end{cases}$$
If $Q= P_1^{e_1} \cdot \ldots \cdot P_k^{e_k}$ is the prime factorization of $Q$ in $\mathbb{F}_q[x]$, then the Jacobi symbol is defined by
$$ \left( \frac{f}{Q} \right) = \prod_{i=1}^k \left( \frac{f}{P_i} \right)^{e_i}.$$
Artin proved the quadratic reciprocity law over function fields, namely that if $A, B \in \mathbb{F}_q[x]$ are non-zero, relatively prime monic polynomials, then
$$ \left( \frac{A}{B} \right) = \left( \frac{B}{A} \right) (-1)^{((q-1)/2) d(A) d(B)}.$$
For $D \in \mathbb{F}_q[x]$, the Dirichlet character $\chi_D$ is defined by $$\chi_D(f) = \left( \frac{D}{f} \right).$$ The $L$-function associated to $\chi_D$ is defined by
$$L(s,\chi_D) = \sum_{f \in \mathcal{M}} \frac{\chi_D(f)}{|f|^s} = \prod_P (1- \chi_D(P) |P|^{-s})^{-1}.$$ This converges for $\text{Re}(s)>1$. Using the change of variables $u=q^{-s}$,
$$\mathcal{L}(u,\chi_D) =  \prod_P (1- \chi_D(P) u^{d(P)})^{-1}.$$ One can show that when $D$ is a non-square polynomial, $\mathcal{L}(u,\chi_D)$ is a polynomial in $u$ of degree at most $d(D)-1$.

When $D$ is a monic square-free polynomial, the completed $L$--function is defined by
$$ \mathcal{L}(u,\chi_D) = (1-u)^{\lambda} \mathcal{L}^{*} (u,\chi_D),$$ where $$ \lambda = 
\begin{cases} 
1 & \mbox{ if } d(D) \text{ even } \\
0 & \mbox{ if } d(D) \text{ odd } 
\end{cases} $$
Then $\mathcal{L}^{*}(u,\chi_D)$ is a polynomial of degree $2 \delta = d(D)-1-\lambda$ and satisfies the functional equation
$$ \mathcal{L}^{*}(u , \chi_D) = (qu^2)^{\delta} \mathcal{L}^{*} (1/(qu),\chi_D).$$ In particular, if $D \in \mathcal{H}_{2g+1}$, then $\mathcal{L}(u,\chi_D)$ is a polynomial of degree $2g$ satisfying the above functional equation.

We can relate the $L$--function to zeta functions of curves. If $C$ is a smooth, projective, geometrically connected curve of genus $g$ over $\mathbb{F}_q$, then the zeta function of $C$ is defined by
$$Z_C (u) = \exp \left( \sum_{r=1}^{\infty} N_r(C) \frac{u^r}{r} \right), $$ where $N_r(C)$ is the number of points on $C$ with coordinates in $\mathbb{F}_{q^r}$. Weil \cite{weil} proved that the zeta function of $C$ is a rational function, equal to
$$ Z_C(u) = \frac{P_C(u)}{(1-u)(1-qu)},$$ where $P_C(u)$ is a polynomial of degree $2g$. The Riemann hypothesis for curves over finite fields was proven by Weil \cite{weil} and states that the zeros of the polynomial $P_C(u)$ all lie on the circle $|u|=q^{-1/2}$. 

When $D$ is monic and square-free, the equation $y^2=D(x)$ defines a projective, connected, hyperelliptic curve. The polynomial $P_{C_D}(u)$ that appears in the zeta function of $C_D$ coincides with the completed $L$--function $\mathcal{L}^{*}(u,\chi_D)$, as proven in Artin's thesis.
\subsection{Preliminary lemmas}
We will quote a number of lemmas we will use in the paper. We assume for simplicity that $q$ is a prime with $q \equiv 1 \pmod 4$.

The following exact formula is an analogue of the approximate functional equation for $L(1/2, \chi_d)$ in the number field setting. 
\begin{lemma}
Let $D \in \mathcal{H}_{2g+1}$. For $k$ an integer, we have the following functional equation:
$$ L\big(\tfrac{1}{2}, \chi_D\big)^k=  \sum_{f \in \mathcal{M}_{\leq kg}} \frac{\chi_D(f) d_k(f)}{\sqrt{|f|}}+   \sum_{f \in \mathcal{M}_{\leq kg-1}} \frac{\chi_D(f) d_k(f)}{\sqrt{|f|}},$$ where $d_k$ is the $k^{\text{th}}$ divisor function.
\label{fe}
\end{lemma}
\begin{proof}
The proof is similar to the proof of the functional equation of $L(1/2, \chi_P)^2$, with $P$ a monic irreducible polynomial in \cite{ak2} and we will omit it. \end{proof}

We also need the following lemma, whose proof can be found in \cite{aflorea}. 
\begin{lemma}
For $f$ a monic polynomial in $\mathbb{F}_q[x]$, we have that
$$ \sum_{D \in \mathcal{H}_{2g+1}} \chi_f(D)= \sum_{C | f^{\infty}} \sum_{h \in \mathcal{M}_{2g+1-2 d(C)}} \chi_f(h) - q \sum_{C | f^{\infty}} \sum_{h \in \mathcal{M}_{2g-1-2 d(C)}} \chi_f(h),$$
where the first sum is over monic polynomials $C$ whose prime factors are among the prime factors of $f$. \label{firstpoint} \end{lemma}
We will now state a version of Poisson summation over function fields. Recall the exponential function introduced in \cite{hayes}. For $a \in \mathbb{F}_q ((1/x ))$,  let
$$ e(a) = e^{2 \pi i a_1/q},$$ where $a_1$ is the coefficient of $1/x$ in the expansion of $a$ (for more details, see \cite{hayes}.) For $\chi$ a general character $\pmod f$, define the generalized Gauss sum as
$$ G(V,\chi) = \sum_{u \pmod f} \chi(u) e \left( \frac{uV}{f} \right).$$
The following Poisson summation formula holds.
\begin{lemma}
Let $f$ be a monic polynomial of degree $n$ in $\mathbb{F}_q[x]$ and $m$ a positive integer. If $d(f)$ is even, then
$$ \sum_{g \in \mathcal{M}_m} \chi_f(g) = \frac{q^m}{|f|} \left[ G(0,\chi_f) + (q-1) \sum_{V \in \mathcal{M}_{\leq n-m-2} } G(V,\chi_f) - \sum_{V \in \mathcal{M}_{n-m-1}} G(V,\chi_f) \right].$$
If $d(f)$ is odd, then
$$ \sum_{g \in \mathcal{M}_m} \chi_f(g) = \frac{q^m}{|f|} \sqrt{q} \sum_{V \in \mathcal{M}_{n-m-1}} G(V,\chi_f).$$
\label{poissonmonic}
\end{lemma}
\begin{proof}
See Proposition $3.1$ in \cite{aflorea}.
\end{proof}
We also need to compute the generalized Gauss sums. The proof is similar to the proof of Lemma $2.3$ in \cite{sound} and we will skip it.
\begin{lemma}
Let $q$ be a prime with $q \equiv 1 \pmod 4$. 
\begin{enumerate}
\item If $(f,g)=1$, then $G(V, \chi_{fg})= G(V, \chi_f) G(V,\chi_g)$.
\item Suppose $V= V_1 P^{\alpha}$ where $P \nmid V_1$.
Then 
 $$G(V , \chi_{P^i})= 
\begin{cases}
0 & \mbox{if }  i \leq \alpha \text{ and } i \text{ odd} \\
\phi(P^i) & \mbox{if }  i \leq \alpha \text{ and } i \text{ even} \\
-|P|^{i-1} & \mbox{if }  i= \alpha+1 \text{ and } i \text{ even} \\
\left( \frac{V_1}{P} \right) |P|^{i-1} |P|^{1/2} & \mbox{if } i = \alpha+1 \text{ and } i \text{ odd} \\
0 & \mbox{if } i \geq 2+ \alpha .
\end{cases}$$ 
\end{enumerate} \label{computeg}
\end{lemma}

 \subsection{Outline of the proof}
We will use the functional equation for $L(1/2,\chi_D)^k$ (with $k=2,3$) as given in Lemma \ref{fe}, and then Lemma \ref{firstpoint} to transform the sum over square-free polynomials $D$ into sums involving monic polynomials. We'll use the Poisson summation formula as in Lemma \ref{poissonmonic} for these sums, getting another summation over monic polynomials $V$. 

We will first focus on the second moment of $L(1/2,\chi_D)$. There will be a main term of size $q^{2g+1} (2g+1)^3$ coming from the contribution of square polynomials $f$ in the functional equation, which we will evaluate in section \ref{main}. Unlike the case of the mean value of $L(1/2,\chi_D)$ in the hyperelliptic ensemble, there will be another secondary main term, which will come from the contribution of square polynomials $V$ (where $V$ is the dual variable in the Poisson summation formula). The secondary main term is of size $q^{2g+1}(2g+1)$, and we will explicitly compute it in section \ref{secondary}. We note that computing the secondary main term is the most delicate part of the proof, and it reduces to exactly evaluating a certain contour integral, which can be done by using a functional equation of the integrand.

We will then evaluate the sum over non-square polynomials $V$ in section \ref{error} and show that it is bounded by $q^{g(1+\epsilon)}$. In section \ref{last} we put together the main term and the secondary main term and check that our answer agrees with the conjecture \eqref{conjpol}.

We use the same methods to evaluate the third moment in section \ref{thirdmom}. Since most of the computations are very similar to the ones carried out before, we will only briefly sketch the proof. The main term corresponding to square polynomials $f$ is of size $q^{2g+1} (2g+1)^6$, and the secondary main term, coming from square $V$, is also of size $q^{2g+1}(2g+1)^6$. We note that evaluating the secondary main term for the third moment again reduces to computing a certain contour integral, which is easier to do than in the second moment case. Here, by simply shifting contours, we get a main term and an error of size $q^{3g/2(1+\epsilon)}$. Bounding the contribution from non-square polynomials $V$ is similar to the method used for the second moment.
 \subsection{Setup of the problem} 
\label{setup} In what follows, $k=2,3$.
Using the functional equation in Lemma \ref{fe} and Lemma \ref{firstpoint} it follows that 
$$\sum_{D \in \mathcal{H}_{2g+1}} L \left( \frac{1}{2}, \chi_D \right)^k = S_{kg}+S_{kg-1},$$ where
$$S_{kg}= \sum_{f \in \mathcal{M}_{\leq kg}} \frac{d_k(f)}{\sqrt{|f|}} \sum_{\substack{ C \in \mathcal{M}_{\leq g} \\ C | f^{\infty}}} \sum_{h \in \mathcal{M}_{2g+1-2d(C)}} \chi_f(h) - q  \sum_{f \in \mathcal{M}_{\leq kg}} \frac{d_k(f)}{\sqrt{|f|}}  \sum_{\substack{ C \in \mathcal{M}_{\leq g-1} \\ C | f^{\infty}}} \sum_{h \in \mathcal{M}_{2g-1-2d(C)}} \chi_f(h). $$ 
Similarly as in \cite{aflorea}, the term in the expression for $S_{kg}$ corresponding to $C \in \mathcal{M}_g$ is bounded by $O(q^{kg/2(1+\epsilon)}).$ Then we rewrite
$$S_{kg} = \sum_{f \in \mathcal{M}_{\leq kg}} \frac{d_k(f)}{\sqrt{|f|}} \sum_{\substack{ C \in \mathcal{M}_{\leq g-1} \\ C | f^{\infty}}} \left(  \sum_{h \in \mathcal{M}_{2g+1-2d(C)}} \chi_f(h) - q\sum_{h \in \mathcal{M}_{2g-1-2d(C)}} \chi_f(h) \right) +O(q^{kg/2(1+\epsilon)}).$$
A similar expression holds for $S_{kg-1}$. We'll focus on $S_{kg}$. Write $S_{kg} = S_{kg,\text{e}} + S_{kg, \text{o}}+O(q^{kg/2(1+\epsilon)})$, where $S_{kg,\text{e}}$ is the sum over polynomials $f$ of even degree less than or equal to $kg$, and $S_{kg,\text{o}}$ the sum over polynomials $f$ of odd degree. When summing over polynomials of even degree, we use the Poisson summation formula in Lemma \ref{poissonmonic} for the sum over $h$, and let $M_{kg}$ be the term corresponding to $V=0$. Note that using Lemma \ref{computeg}, $G(0,\chi_f)$ is nonzero if and only if $f$ is a square, in which case $G(0,\chi_f) = \phi(f)$.  Write $S_{kg, \text{e}} = M_{kg} + S_{kg,\text{e}}(V \neq 0)$, where 
\begin{equation*}
  M_{kg} = q^{2g+1}  \Big( 1-\frac{1}{q} \Big) \sum_{\substack{f \in \mathcal{M}_{\leq kg} \\ f = \square}} \frac{d_k(f)}{|f|^{\frac{3}{2}}}  \phi(f)\sum_{\substack{C \in \mathcal{M}_{\leq g-1} \\ C | f^{\infty}}} \frac{1}{|C|^2} ,
 \end{equation*}
  \begin{align} S_{kg,\text{e}}(V \neq 0) &= q^{2g+1} \sum_{\substack{f \in \mathcal{M}_{\leq kg} \\ d(f) \text{ even}}} \frac{d_k(f)}{|f|^{\frac{3}{2}}}\sum_{\substack{C \in \mathcal{M}_{\leq g-1} \\ C | f^{\infty}}} \frac{1}{|C|^2} \bigg[ (q-1) \sum_{ V \in \mathcal{M}_{\leq d(f)-2g-3+2d(C)} } G(V, \chi_f) - \sum_{V \in \mathcal{M}_{d(f)-2g-2+2d(C)}} G(V,\chi_f) \nonumber \\
 & - \frac{q-1}{q} \sum_{V \in \mathcal{M}_{\leq d(f)-2g-1+2d(C)} } G(V, \chi_f) + \frac{1}{q} \sum_{V \in \mathcal{M}_{d(f)-2g+2d(C)}} G(V,\chi_f) \bigg] .\label{s1e} \end{align}
 Again using the Poisson summation formula in \ref{poissonmonic},
 \begin{equation}
 S_{kg,\text{o}} = q^{2g+1}  \sqrt{q} \sum_{\substack{f \in \mathcal{M}_{\leq kg} \\ d(f) \text{ odd}}} \frac{d_k(f)}{|f|^{\frac{3}{2}}} \sum_{\substack{C \in \mathcal{M}_{\leq g-1} \\ C | f^{\infty}}} \frac{1}{|C|^2}  \left( \sum_{V \in \mathcal{M}_{d(f)-2g-2+2d(C)}} G(V, \chi_f)  - \frac{1}{q} \sum_{V \in \mathcal{M}_{d(f)-2g+2d(C)}} G(V, \chi_f) \right). \label{odd} \end{equation}
  
  In equation \eqref{s1e}, we write the sum over $V$ as the sum over square $V$ plus the sum over non-square $V$. Let $S_{kg,\text{e}}(V \neq 0)=S_{kg}(V=\square) + S_{kg,\text{e}}(V \neq \square)$.
  
  When $V$ is a square, write $V=l^2$. Using equation \eqref{s1e}, we rewrite
 \begin{align}
  S_{kg}(V= \square) &= q^{2g+1}\sum_{\substack{f \in \mathcal{M}_{\leq kg} \\ d(f) \text{ even}}} \frac{d_k(f)}{|f|^{\frac{3}{2}}}\sum_{\substack{C \in \mathcal{M}_{\leq g-1} \\ C | f^{\infty}}} |C|^{-2} \bigg[ (q-1)  \sum_{ l \in \mathcal{M}_{\leq \frac{d(f)}{2}-g-2+d(C)} } G(l^2, \chi_f) - \sum_{l \in \mathcal{M}_{\frac{d(f)}{2}-g-1+d(C)}} G(l^2,\chi_f) \nonumber \\
  &- \frac{q-1}{q}  \sum_{ l \in \mathcal{M}_{\leq \frac{d(f)}{2}-g-1+d(C)} } G(l^2, \chi_f) + \frac{1}{q} \sum_{l \in \mathcal{M}_{\frac{d(f)}{2}-g+d(C)}} G(l^2,\chi_f) \bigg] .\label{s1s} \end{align}   Let $S_k(V= \square) =S_{kg}(V=\square)+ S_{kg-1}(V=\square)$ (where $S_{kg-1}(V=\square)$ is defined in the same way as $S_{kg}(V=\square)$). We'll show that $S_k(V=\square)$ (which is the secondary main term) is of size $g q^{2g+1}$ when $k=2$ and of size $g^6 q^{2g+1}$ when $k=3$.  Define $S_{kg}(V \neq \square) = S_{kg,\text{o}} + S_{kg,\text{e}}(V \neq \square)$, with $ S_{kg,\text{o}}$ given by \eqref{odd} and $S_{kg,\text{e}}(V \neq \square)$ the sum over non-square polynomials $V$ in \eqref{s1e}. Similarly define $S_{kg-1}(V \neq \square)$. We'll bound $S_{kg}(V \neq \square)$ and $S_{kg-1}(V \neq \square)$ by $O(q^{kg/2(1+\epsilon)})$.
\label{setup}
\section{Main term} 
\label{main}
In the next four sections, we will concentrate on the second moment of $L(1/2,\chi_D)$. Here, we focus on evaluating the main term corresponding to the contribution of square polynomials $f$. 
Recall that
\begin{equation}    M_{2g} = q^{2g+1}  \Big( 1-\frac{1}{q} \Big) \sum_{\substack{f \in \mathcal{M}_{\leq 2g} \\ f = \square}} \frac{d_2(f)}{|f|^{\frac{3}{2}}}  \phi(f)\sum_{\substack{C \in \mathcal{M}_{\leq g-1} \\ C | f^{\infty}}} \frac{1}{|C|^2}. \label{principal} \end{equation} A similar expression holds for $M_{2g-1}$. 
The main term $M_{2g}+M_{2g-1}$ is given by the following lemma.
\begin{lemma}
Using the same notation as before, we have
$$M_{2g}+M_{2g-1} = \frac{q^{2g+1}}{\zeta(2)} P_1(2g+1) + O(q^{g(1+\epsilon)}),$$ where $P_1$ is the polynomial of degree $3$ given by \eqref{p1x}.
\label{lemma:lemmamainterm}
\end{lemma}
To prove this, we express $M_{2g}$ and $M_{2g-1}$ as contour integrals and then evaluate them. We do so in the next lemma.
\begin{lemma}
With the same notation as before, we have
$$M_{2g} = \frac{q^{2g+1}}{\zeta(2)} \frac{1}{2 \pi i} \oint_{|u|=r_1} \frac{(1-qu^2) \mathcal{B}(u)}{(1-qu)^4  (qu)^g} \, \frac{du}{u} + O(q^{g \epsilon}),$$ and
$$ M_{2g-1} = \frac{q^{2g+1}}{\zeta(2)} \frac{1}{2 \pi i} \oint_{|u|=r_1} \frac{(1-qu^2) \mathcal{B}(u)}{(1-qu)^4  (qu)^{g-1}} \, \frac{du}{u} + O(q^{g \epsilon}), $$ where \begin{equation}
\mathcal{B}(u) = \prod_P \left( 1+ \frac{ u^{d(P)}(u^{d(P)}-3)}{(|P|+1)(1+u^{d(P)})} \right), \label{bu} 
\end{equation} and $r_1<1/q$. \label{lemma:maintermlemma}
\end{lemma}
\begin{remark}
Note that $\mathcal{B}(u)$ converges for $|u|<1$.
\end{remark}
\begin{proof}
In equation \eqref{principal}, write $f=l^2$, with $l \in \mathcal{M}_m$. Note that $C | f^{\infty}$ if and only if $C | l^{\infty}$. Similarly as in \cite{aflorea}, we have
$$\sum_{\substack{C \in \mathcal{M}_{\leq g-1} \\ C | f^{\infty}}} |C|^{-2} =\sum_{\substack{C \in \mathcal{M}_{\leq g-1} \\ C | l^{\infty}}} |C|^{-2}= \prod_{P|l} \left( 1- \frac{1}{|P|^2} \right)^{-1} + O(q^{-g(2-\epsilon)}).$$  Since $\phi(l^2)/|l|^2 = \prod_{P|l} (1-|P|^{-1})$ and $(1-1/q)^{-1} = \zeta(2)$, we have
\begin{equation}
M_{2g} = \frac{q^{2g+1}}{\zeta(2)} \sum_{l \in \mathcal{M}_{\leq g}} \frac{d_2(l^2)}{|l| \displaystyle \prod_{P|l} (1+|P|^{-1})} + O(q^{g \epsilon}) .\label{interm1} \end{equation}
Let
$$ \mathcal{A}(u) = \sum_{l \in \mathcal{M}}  \frac{d_2(l^2)}{\displaystyle \prod_{P|l} (1+|P|^{-1})} u^{d(l)}.$$ By multiplicativity, we can write
$$\mathcal{A}(u)= \frac{\mathcal{Z}(u)^3}{\mathcal{Z}(u^2)} \mathcal{B}(u),$$ with $\mathcal{B}(u)$ given by \eqref{bu}. Note that $\mathcal{Z}(u)=(1-qu)^{-1}$ and $\mathcal{Z}(u^2)=(1-qu^2)^{-1}$, so
\begin{equation}
\mathcal{A}(u)= \frac{1-qu^2}{(1-qu)^3} \mathcal{B}(u). \label{au}
\end{equation}
Now we will use the following remark, which is the function field analogue of Perron's formula. If the power series $\sum_{f \in \mathcal{M}} a(f) u^{d(f)}$ converges absolutely for $|u| \leq R < 1$, then
\begin{equation}
\sum_{f \in \mathcal{M}_{\leq k}} a(f) = \frac{1}{2 \pi i} \oint_{|u|=R} \left( \sum_{f \in \mathcal{M}} a(f) u^{d(f)} \right) \frac{u^{-k}}{1-u} \, \frac{du}{u}. \label{perron} \end{equation} Using this in \eqref{interm1} gives

$$ M_{2g} =  \frac{q^{2g+1}}{\zeta(2)} \frac{1}{2 \pi i} \oint_{|u|=r_1} \frac{(1-qu^2) \mathcal{B}(u)}{(1-qu)^4  (qu)^g} \, \frac{du}{u} + O(q^{g \epsilon}),$$ where $r_1<1/q$. 
We can similarly express $M_{2g-1}$, which finishes the proof of Lemma \ref{lemma:maintermlemma}. 
\end{proof}
\begin{proof}[Proof of Lemma \ref{lemma:lemmamainterm}]
In Lemma \ref{lemma:maintermlemma}, note that the integrand $((1-qu^2) \mathcal{B}(u))/(u(1-qu)^4 (qu)^g)$ has a pole of order $4$ at $u=1/q$. Since $\mathcal{B}(u)$ converges absolutely for $|u|<1$, we can write
$$ \frac{1}{2 \pi i} \oint_{|u|=r_1} \frac{(1-qu^2) \mathcal{B}(u)}{(1-qu)^4  (qu)^g} \, \frac{du}{u} = - \text{Res}(u=1/q) + \frac{1}{2 \pi i} \oint_{|u|=r_2} \frac{(1-qu^2) \mathcal{B}(u)}{(1-qu)^4  (qu)^g} \, \frac{du}{u},$$ where $r_2=q^{- \epsilon}$. We can explicitly compute the residue at $u=1/q$, and we can bound the integral on the right-hand side above by
$$ \left| \frac{1}{2 \pi i} \oint_{|u|=r_2} \frac{(1-qu^2) \mathcal{B}(u)}{(1-qu)^4 u (qu)^g} \, du \right| \ll q^{-g(1-\epsilon)}.$$ We can similarly express $M_{2g-1}$ in terms of the residue of the integrand $ ((1-qu^2) \mathcal{B}(u))/(u(1-qu)^4 (qu)^{g-1})$ at $u=1/q$. Computing the residues at $u=1/q$ gives
\begin{equation} M_{2g}+M_{2g-1} = \frac{q^{2g+1}}{\zeta(2)} P_1(2g+1) + O(q^{g(1+\epsilon)}), \label{firstterm} \end{equation}  where $P_1(x)$ is a polynomial of degree $3$. We compute it explicitly as 
\begin{align}
P_1(x) &= x^3  \frac{\bc(1/q) (1-q^{-1}) }{24} + x^2 \left[ \frac{\bc(1/q) (1+q^{-1})}{4} - \frac{\bc'(1/q) (1-q^{-1})}{4q}  \right] \nonumber \\
&+ x \left[ \frac{11 \bc(1/q) (1-q^{-1})}{24} + \frac{3 \bc'(1/q)(1-q^{-1})}{2q}  - \frac{2 \bc'(1/q)}{q} + \frac{ \bc''(1/q) (1-q^{-1})}{2q^2}  \right] \nonumber \\
&+ \frac{ \bc(1/q) (1+q^{-1})}{4} - \frac{ \bc' (1/q) (1-q^{-1})}{4q} + \frac{ 2 \bc'(1/q)}{q^2} +\frac{2 \bc''(1/q)}{q^3} - \frac{ \bc^{(3)} (1/q) (1-q^{-1})}{3q^3}.  \label{p1x} \end{align}
\end{proof}

\section{Secondary main term}
\label{secondary}
In this section, we will evaluate the secondary main term $S_2(V=\square)$ coming from the contribution of square polynomials $V$. Recall from subsection \ref{setup} that $S_2(V=\square)= \sg+ \sgm$, where $\sg$ is given by equation \eqref{s1s}.
We will prove the following.
\begin{lemma}
Using the same notation as before, we have that
$$S_2(V=\square) = \frac{q^{2g+1}}{\zeta(2)} P_2(2g+1)+O(q^{g(1+\epsilon)}), $$ where $P_2(x)$ is a linear polynomial which can be computed explicitly (see formula \eqref{p2x}.)
\label{lemma:sec} 
\end{lemma}

\subsection{A few lemmas} 
\label{lemmas}
To prove Lemma \ref{lemma:sec}, we will first prove the following results.
\begin{lemma}
Let $V$ be a monic polynomial in $\fqx$. For $|z|>1/q^2$, let $$\mathcal{M}(V;z,w)= \sum_{f \in \mathcal{M}} \frac{d_2(f) G(V,\chi_f) }{\sqrt{|f|} \displaystyle \prod_{P|f} \left(1-\frac{1}{|P|^2 z^{d(P)}} \right)} w^{d(f)}.$$ 
\begin{enumerate}[(a)]
\item
We have $$ \mathcal{M}(V;z,w) = \mathcal{L}(w,\chi_V)^2 \prod_{P} \mathcal{M}_P(V;z,w),$$ where
$$ \mathcal{M}_P(V;w,u) = 
\begin{cases} 
1+ \frac{2 \left( \frac{V}{P} \right) w^{d(P)}}{|P|^2 z^{d(P)} -1} + w^{2d(P)} - \frac{4w^{2 d(P)}}{1-\frac{1}{|P|^2 z^{d(P)}}} + \frac{2 \left( \frac{V}{P} \right) w^{3 d(P)}}{1-\frac{1}{|P|^2 z^{d(P)}}} & \mbox{if } P \nmid V \\
1+ \left(1-\frac{1}{|P|^2 z^{d(P)}}\right)^{-1} \sum_{b=1}^{\infty} \frac{d_2(P^b) G(V, \chi_{P^b})}{|P|^{b/2}} w^{b d(P)} & \mbox{if } P|V
\end{cases} $$
\item
If $V=l^2$, with $l \in \mathcal{M}$, then
$$ \mathcal{M}(l^2;z,w)= 
Z(w)^2 \prod_{P} \mathcal{R}_P(l^2;z,w),$$ where
$$  \mathcal{R}_P(l^2;z,w) = 
\begin{cases}  
1+ \frac{2  w^{d(P)}}{|P|^2 z^{d(P)} -1} + w^{2d(P)} - \frac{4w^{2 d(P)}}{1-\frac{1}{|P|^2 z^{d(P)}}} + \frac{2  w^{3 d(P)}}{1-\frac{1}{|P|^2 z^{d(P)}}}& \mbox{if } P \nmid l \\
(1-w^{d(P)})^2 \left( 1+ \left(1-\frac{1}{|P|^2 z^{d(P)}}\right)^{-1} \sum_{b=1}^{\infty} \frac{d_2(P^b) G(l^2, \chi_{P^b})}{|P|^{b/2}} w^{b d(P)}\right) & \mbox{if } P|l \end{cases} $$
Moreover, $\prod_P \mathcal{R}_P(l^2;z,w)$ converges absolutely for $|w|<q|z|$ and $|w|<q^{-1/2}$. \end{enumerate} \label{mv}
\end{lemma}
\begin{proof}
We use the fact that $G(V,\chi_f)$ is multiplicative as a function of $f$ and then we manipulate Euler products.
\end{proof}
\begin{lemma}
Let $$\mathcal{R}(z,w)= \sum_{l \in \mathcal{M}} z^{d(l)} \prod_P \mathcal{R}_P(l^2;z,w),$$ with $\mathcal{R}_P (l^2;z,w)$ defined in Lemma \ref{mv}. Then
$$ \mathcal{R}(z,w) = \mathcal{Z}(z) \mathcal{Z}(qw^2z)  \mathcal{Z} \left( \frac{1}{q^2z} \right) \mathcal{F}(z,w) ,$$ where $\mathcal{F}(z,w) = \prod_P  \mathcal{F}_P(z,w)$, with \newline
$ \mathcal{F}_P(z,w)= (1-w^{d})^2  \left( 1- \frac{1-2|P|^2 (wz)^d-2|P|(w^2z)^d+2|P|^2(wz^2)^d+(-2|P|^3+3|P|^2)(w^2z^2)^d+|P|^2(w^4z^2)^d-|P|^3(w^4z^3)^d}{|P|^2z^d(1-|P|w^{2d}z^d)} \right)$ (here $d$ stands for $d(P)$.) \newline Moreover, $\mathcal{F}(z,w)$ is absolutely convergent for $|z|>1/q, |w|<1/\sqrt{q}, |wz|<1/q$ and $|w^2z|< 1/q^2$. \label{rz}
\end{lemma}
\begin{proof}
We use Lemma \ref{mv} and then manipulate Euler products.
\end{proof}

\begin{lemma}
Let $\alpha(z) = \frac{\frac{1}{q} \frac{d}{dw} \mathcal{F}(z,w) \rvert_{w=1/q}}{\mathcal{F}(z,1/q)}$, with $\mathcal{F}(z,w)$ defined in the previous lemma, and let $\mathcal{F}(z) = \mathcal{F} \left(z, \frac{1}{q} \right)$.

(a) We have $$\mathcal{F} (z) =  \prod_P \left(1-\frac{1}{|P|} \right)^2 \left(1+ \frac{2}{|P|}+\frac{1}{|P|^3} - \frac{1}{|P|^2} (z^{d(P)}+z^{-d(P)}) \right),$$ and $\mathcal{F}(z)=\mathcal{F}(1/z)$.

(b) We have $$ \alpha(z) =   \sum_P  \frac{ 2 d(P) (|P|^2+ z^{d(P)}(-3|P|^3-3|P|+|P|^2)+ z^{2d(P)}(|P|^4-|P|^3+4|P|^2-|P|+2)+z^{3d(P)}(|P|^2-2|P|))}{(|P|-1)(z^{d(P)}-|P|)(|P|-z^{d(P)}-2|P|^2z^{d(P)}-|P|^3z^{d(P)}+|P|z^{2d(P)})},$$  and
 $$\alpha(1/z) = \alpha(z) -\frac{2(1+z)}{1-z} - 4z \frac{\mathcal{F}'(z)}{\mathcal{F}(z)}.$$ 
\label{sim}
\end{lemma}
\begin{proof}
The first part follows directly by computation from Lemma \ref{rz}. For the second part, we rewrite
$$ \alpha(z) = \sum_P \frac{2 d(P)}{\frac{|P|}{z^{d(P)}} -1} - g(z)-4h(z) = \frac{2z}{1-z} - g(z)-4h(z), $$ with
$$g(z) = \sum_P \frac{ d(P) (6z^{d(P)}-4|P|z^{d(P)}+6|P|^2z^{d(P)} -2|P|-2|P|z^{2d(P)})}{(|P|-1)(z^{d(P)}+2|P|^2z^{d(P)}+|P|^3z^{d(P)}-|P|-|P|z^{2d(P)})},$$ and
$$ h(z) = \sum_P \frac{ d(P) |P| z^{2d(P)}}{z^{d(P)}+2|P|^2z^{d(P)}+|P|^3z^{d(P)}-|P|-|P|z^{2d(P)}}.$$  
Note that from the definition of $g(z)$ and $h(z)$ and from the expression for $\mathcal{F}(z)$, we have $g(z)=g(1/z)$ and 
$$h(1/z) = h(z) + z \frac{\mathcal{F}'(z)}{\mathcal{F}(z)}.$$  Combining these, the conclusion now follows.
\end{proof}
\subsection{Proof of Lemma \ref{lemma:sec}}
\label{pf1}
We now begin the proof of Lemma \ref{lemma:sec}. Recall the formula \eqref{s1s} for $S_{2g}(V=\square)$. Using \eqref{perron} twice, we have
$$ S_{2g}(V=\square) =  q^{2g+1} \frac{1}{2 \pi i} \oint_{|z|=r_1}   \sum_{\substack{f \in \mathcal{M}_{\leq 2g} \\ d(f) \text{ even}}} \frac{1}{|f|} \sum_{\substack{C \in \mathcal{M}_{\leq g-1} \\ C|f^{\infty}}} \frac{(qz-1)}{(1-z) z^{\frac{d(f)}{2}-g+d(C)}} \left( \sum_{l \in \mathcal{M}} z^{d(l)}   \frac{d_2(f) G(l^2, \chi_f)}{\sqrt{|f|}} \right)  \left(1-\frac{1}{qz} \right)\, dz, $$ where we pick $r_1 = q^{-1-\epsilon}$. We can extend the sum over $C | f^{\infty}$ with $d(C) \leq g-1$ to include all polynomials $C|f^{\infty}$ similarly as in \cite{aflorea}, at the expense of a term of size $q^{g(1+\epsilon)}$. Since
$$ \sum_{C | f^{\infty}} \frac{1}{|C|^2 z^{d(C)}} = \prod_{P|f} \left(1- \frac{1}{|P|^2 z^{d(P)}} \right)^{-1},$$ we have
\begin{align*}
   S_{2g}(V=\square) &=  q^{2g+1} \frac{1}{2 \pi i} \oint_{|z|=r_1} \frac{(qz-1)z^g}{1-z}  \sum_{l \in \mathcal{M}} z^{d(l)}\sum_{\substack{f \in \mathcal{M}_{\leq 2g} \\ d(f) \text{ even}}} \frac{1}{|f|z^{\frac{d(f)}{2}}}   \frac{d_2(f) G(l^2, \chi_f)}{\sqrt{|f|} \prod_{P|f} \left(1- \frac{1}{|P|^2 z^{d(P)}} \right) } \left(1-\frac{1}{qz} \right)\, dz \\
   &+ O(q^{g(1+\epsilon)}). \end{align*}
Using Lemma \ref{mv}, equation \eqref{perron} and Lemma \ref{rz}, it follows that
$$  S_{2g}(V=\square) =  q^{2g+1} \left( \frac{1}{2 \pi i} \right)^2 \oint_{|z|=r_1} \oint_{|w|=r_2} \frac{(qz-1)z^g (q^2w^2z)^{-g}}{(1-z)w(1-qw)^2 (1-q^2w^2z)}  \mathcal{R}(z,w)\left(1-\frac{1}{qz} \right) \, dw \, dz + O(q^{g(1+\epsilon)}),$$ where recall that $r_1=q^{-1-\epsilon}$ and $r_2<1/q$. Using Lemma \ref{rz} again, we have
$$ S_{2g} (V= \square) = -q^{2g+1} \left( \frac{1}{2 \pi i} \right)^2 \oint_{|z|=r_1} \oint_{|w|=r_2} \frac{z^g (q^2w^2z)^{-g} \mathcal{Z}(1/(q^2z)) \mathcal{F}(z,w)}{(1-z)w(1-qw)^2 (1-q^2w^2z)^2} \left(1-\frac{1}{qz} \right) \, dw \, dz + O(q^{g(1+\epsilon)}).$$ From Lemma \ref{rz}, $\mathcal{Z}(1/(q^2z))\mathcal{F}(z,w)$ is absolutely convergent for $|w|<1/\sqrt{q}, |wz|<1/q$ and $|w^2z|<1/q^2$, so in the integral above we can shift the contour $|z|=q^{-1-\epsilon}$ to $|z|=q^{\epsilon-1}$ without encountering any poles. Then we have
$$ S_{2g} (V= \square) = - q   \left( \frac{1}{2 \pi i} \right)^2 \oint_{|z|=r_1} \oint_{|w|=r_2} \frac{\mathcal{F}(z,w)}{(1-z)(1-qw)^2(1-q^2w^2z)^2 w^{2g+1}} \, dw \, dz + O(q^{g(1+\epsilon)}),$$
where $r_1=q^{\epsilon-1}$ and $r_2<1/q$. Enlarging the contour of integration $|w|=r_2$ to $|w|= q^{-1/2-\epsilon}$, we encounter a double pole at $w=1/q$, and the double integral $ \left( \frac{1}{2 \pi i} \right)^2 \oint_{|z|=r_1} \oint_{|w|=q^{-1/2-\epsilon}}$ will be bounded by $q^{g(1+\epsilon)}$. We compute the residue at $w=1/q$ and using the notation from Lemma \ref{sim} we have
\begin{equation*}
\sg = - q^{2g+1} \oint_{|z|=\frac{q^{\epsilon}}{q}} \frac{\mathcal{F}(z)}{(1-z)^3} \left( 2g+1 - \frac{4z}{1-z} - \alpha(z) \right) \, dz + O(q^{g(1+\epsilon)}).
\end{equation*}
Similarly
\begin{equation*}
\sgm=  - q^{2g+1}  \oint_{|z|=\frac{q^{\epsilon}}{q}} \frac{z\mathcal{F}(z)}{(1-z)^3} \left( 2g-1 - \frac{4z}{1-z} - \alpha(z) \right) \, dz + O(q^{g(1+\epsilon)}).
\end{equation*}
Combining the above and using the fact that $S_2(V= \square) = \sg+\sgm$, it follows that
\begin{align}
S_2(V=\square)= -q^{2g+1} \frac{1}{2 \pi i}  \oint_{|z|=\frac{q^{\epsilon}}{q}} \frac{\mathcal{F}(z) (1+z)}{(1-z)^3} \left( 2g+1- \frac{6z}{1-z^2}-\frac{2z^2}{1-z^2} - \alpha(z) \right)  \, dz \label{formula}  +O(q^{g(1+\epsilon)}) .\nonumber
\end{align} 
 From Lemma \ref{sim}, note that $\mathcal{F}(z)$ is analytic for $1/q<|z|<q$. We can compute the integral above exactly using the symmetry properties of $\mathcal{F}$ and $ \alpha$, by making the change of variables $z=1/u$. Combining the functional equations for $\mathcal{F}(z)$ and $\alpha(z)$ as given in Lemma \ref{sim} and using the fact that
 $$ \frac{1}{2 \pi i}  \oint_{|z|=\frac{q^{\epsilon}}{q}} \frac{z(1+z) \mathcal{F}'(z)}{(1-z)^3}  \, dz= - \frac{1}{2\pi i}  \oint_{|z|=\frac{q^{\epsilon}}{q}} \frac{\mathcal{F}(z)(z^2+4z+1)}{(1-z)^4} \, dz ,$$ it follows that
 \begin{align*}
 & \frac{1}{2 \pi i}  \oint_{|z|=\frac{q^{\epsilon}}{q}} \frac{\mathcal{F}(z) (1+z)}{(1-z)^3} \left( 2g+1- \frac{6z}{1-z^2}-\frac{2z^2}{1-z^2} - \alpha(z) \right)  \, dz \\
 & = - \frac{1}{2 \pi i} \oint_{|z|=\frac{q}{q^{\epsilon}}} \frac{\mathcal{F}(z) (1+z)}{(1-z)^3} \left( 2g+1- \frac{6z}{1-z^2}-\frac{2z^2}{1-z^2} - \alpha(z) \right)  \, dz.
 \end{align*}
 Note that in the annulus between $|z|=q^{\epsilon-1}$ and $|z|=q^{1-\epsilon}$ there is only one pole of the integrand at $z=1$. Hence from the identity above, we can explicitly evaluate the integral as
$$  \frac{1}{2 \pi i} \oint_{|z|=\frac{q^{\epsilon}}{q}} \frac{\mathcal{F}(z) (1+z)}{(1-z)^3} \left( 2g+1- \frac{6z}{1-z^2}-\frac{2z^2}{1-z^2} - \alpha(z) \right)  \, dz  = -\frac{ \text{Res}(z=1)}{2}.$$
Then
  \begin{equation*}
  S_2(V=\square) =  \frac{q^{2g+1}}{2} \text{Res}(z=1) + O(q^{g(1+\epsilon)}). 
 \end{equation*} Computing the residue at $z=1$ gives that
 $$ S_2(V=\square)  = \frac{q^{2g+1}}{\zeta(2)} P_2(2g+1) + O(q^{g(1+\epsilon)}),$$ where $P_2(2g+1) = 2 \zeta(2) \text{Res}(z=1),$ and we compute it explicitly as
 \begin{align}
 P_2(x)&= -x \frac{\zeta(2)}{2} (\mathcal{F}'(1)+\mathcal{F}''(1)) \label{p2x} \\
 &- \zeta(2) \left[ 2\mathcal{F}'(1)+4\mathcal{F}''(1)+\mathcal{F}^{(3)}(1)+\frac{\alpha(1)(\mathcal{F}'(1)+\mathcal{F}''(1))}{2} +  \frac{\alpha'(1)(\mathcal{F}(1)+\mathcal{F}'(1))}{2} +\frac{ \alpha''(1) \mathcal{F}(1) }{2} \right]. \nonumber
 \end{align}  This finishes the proof of Lemma \ref{lemma:sec}. 
\section{Evaluating the error from non-square $V$}
In this section, we will bound $S_{2g}(V \neq \square)$ and $S_{2g-1}(V \neq \square)$. Recall that $S_{2g}(V \neq \square) = S_{2g,\text{o}} + S_{2g,\text{e}}(V \neq \square)$, with $ S_{2g,\text{o}}$ given by \eqref{odd} and $S_{2g,\text{e}}(V \neq \square)$ the sum over non-square polynomials $V$ in \eqref{s1e}. We will prove the following.
\begin{lemma}
Using the same notation as before, we have
$$S_{2g}(V \neq \square) \ll q^{g(1+\epsilon)},$$ and
$$S_{2g-1}(V \neq \square) \ll q^{g(1+\epsilon)}.$$
\label{errorsquare}
\end{lemma}
\begin{proof}
In equation \eqref{odd}, write $S_{2g,\text{o}}$ as a difference of two terms, and let $S_{1,\text{o}}$ denote the first term and $S_{2,\text{o}}$ the second. We will bound the term $S_{1,\text{o}}$ and then bounding $S_{2,\text{o}}, S_{2g,\text{e}}(V \neq \square)$ will follow similarly. We use the fact that
$$ \sum_{\substack{C \in \mathcal{M}_i \\ C | f^{\infty} }} \frac{1}{|C|^2} = \frac{1}{2 \pi i} \oint_{|u|=r_1} \frac{1}{q^{2i}u^{i+1} \prod_{P|f} (1-u^{d(P)})}, $$ where $r_1<1$. If we let $d(f)=n$ and $d(C)=i$, then 
\begin{equation}
S_{1,\text{o}} = q^{2g+1} \sqrt{q} \frac{1}{2 \pi i} \oint_{|u|=r_1} \sum_{\substack{n=0 \\ n \text{ odd}}}^{2g} q^{-n}   \sum_{i=0}^g \frac{1}{q^{2i} u^{i+1}} \sum_{V \in \mathcal{M}_{n-2g-2+2i}} \sum_{f \in \mathcal{M}_n} \frac{ d_2(f) G(V,\chi_f)|f|^{-1/2}}{\displaystyle \prod_{P|f} (1-u^{d(P)})} \, du . \label{err1}
\end{equation} 
We express the sum over $f$ as a contour integral and use Lemma \ref{mv} to write
$$  \sum_{f \in \mathcal{M}_n}  \frac{d_2(f) G(V,\chi_f) |f|^{-1/2}}{ \displaystyle \prod_{P|f} (1- u^{d(P)})} = \frac{1}{2 \pi i} \oint_{|w|=r_2} \frac{\mathcal{L}(w,\chi_V)^2 \prod_P \mathcal{M}_P(V;1/(q^2u),w)}{w^{n+1}} \, dw.$$ From Lemma \ref{mv}, note that $\prod_P \mathcal{M}_P(V;1/(q^2u),w)$ converges for $|wu|<1/q, |w|<1/\sqrt{q}$ and $|u|<1$. We pick $r_1=q^{-\epsilon}$ and $r_2=q^{-1/2-\epsilon}$. Let $k$ be the least integer such that $ r_1^k r_2<1/q$. Then we can write
$$ \prod_P \mathcal{M}_P(V;1/(q^2u),w) = \mathcal{L}(wu,\chi_V)^2 \mathcal{L}(wu^2,\chi_V)^2 \cdot \ldots \cdot \mathcal{L}(wu^{k-1}, \chi_V)^2 \mathcal{B}(V;w,u),$$ where $\mathcal{B}(V;w,u)$ is given by a converging Euler product. Hence we bound
\begin{equation}
 \left|  \sum_{f \in \mathcal{M}_n}  \frac{d_2(f) G(V,\chi_f) |f|^{-1/2}}{ \displaystyle \prod_{P|f} (1- u^{d(P)})} \right| \ll q^{n/2(1+\epsilon)} \left| \mathcal{L}(w,\chi_V) \cdot \ldots \cdot \mathcal{L}(wu^{k-1},\chi_V) \right|^2. \label{err2} \end{equation} Note that the degree of $V$ is odd, so $V$ can't be a square. Using theorem $3.3$ in \cite{altug} and the remarks in the proof of Lemma $7.1$ in \cite{aflorea}, it follows that
$$ | \mathcal{L}(wu^j,\chi_V) | \ll e^{ \frac{n-2g+2i}{ 2 \log_q(n/2-g+i)} + 4\sqrt{q(n-2g+2i)}}, $$ for $j \in \{0, \ldots, k-1 \}$. Using the bound above and combining equations \eqref{err2} and \eqref{err1}, we get that
$$ S_{1,\text{o}} \ll q^{g(1+\epsilon)} e^{\frac{2gk}{\log_q(g)} + 8 \sqrt{2gq} k } \ll q^{g(1+\epsilon)}.$$ Hence $S_{2g}(V \neq \square) \ll q^{g(1+\epsilon)}.$
\end{proof}

\label{error}
\section{Proof of Theorem \ref{mresult}}
Now we put together the results from the previous sections. Combining Lemma \ref{lemma:lemmamainterm}, equation \eqref{p1x}, Lemma \ref{lemma:sec}, equation \eqref{p2x} and Lemma \ref{errorsquare}, it follows that
$$ \sum_{D \in \mathcal{H}_{2g+1}} L \left(\frac{1}{2}, \chi_D \right)^2= \frac{q^{2g+1}}{\zeta(2)} P(2g+1) + O(q^{g(1+\epsilon)}),$$ where $P(x)=P_1(x)+P_2(x)$ is the following degree $3$ polynomial
\begin{align}
P(x) & = x^3  \frac{\bc(1/q) (1-q^{-1}) }{24} + x^2 \left[ \frac{\bc(1/q) (1+q^{-1})}{4} - \frac{\bc'(1/q) (1-q^{-1})}{4q}  \right] \nonumber \\
&+ x \left[ \frac{11 \bc(1/q) (1-q^{-1})}{24} + \frac{3 \bc'(1/q)(1-q^{-1})}{2q}  - \frac{2 \bc'(1/q)}{q} + \frac{ \bc''(1/q) (1-q^{-1})}{2q^2}  - \frac{\zeta(2)}{2} (\mathcal{F}'(1)+\mathcal{F}''(1)) \right] \nonumber \\
&+ \frac{ \bc(1/q) (1+q^{-1})}{4} - \frac{ \bc' (1/q) (1-q^{-1})}{4q} + \frac{ 2 \bc'(1/q)}{q^2} +\frac{2 \bc''(1/q)}{q^3} - \frac{ \bc^{(3)} (1/q) (1-q^{-1})}{3q^3}  \nonumber \\
&- \zeta(2) \left[ 2\mathcal{F}'(1)+4\mathcal{F}''(1)+\mathcal{F}^{(3)}(1)+ \frac{\alpha'(1)(\mathcal{F}(1)+\mathcal{F}'(1))}{2}+ \frac{\alpha(1)(\mathcal{F}'(1)+\mathcal{F}''(1))}{2}+\frac{\mathcal{F}(1) \alpha''(1)}{2} \right] \label{px}
\end{align}
\begin{remark}
We can check that the answer above matches the conjectured result in \eqref{conjpol}. 
For a polynomial $Q$, let $[x^i]Q$ denote the coefficient of $x^i$ in the polynomial $Q$.  We will check that $[x^i]P=[x^i]R$ for all $i  \in \{0,1,2,3\}$, with $R$ given by \eqref{conjpol}.
\begin{enumerate}[(i)]
\item We have that $[x^3] P =  \frac{\bc(1/q) (1-q^{-1}) }{24}$ and $[x^3] R = \frac{A(0,0)}{24}$. Using Lemma \ref{lemma:maintermlemma},
$$ \bc(1/q) = \prod_P \left(1 + \frac{1-3|P|}{|P|(|P|+1)^2} \right).$$ The identity then easily follows upon noticing that $\bc(1/q)   \prod_P (1-|P|^{-2}) = A(0,0)$ and that $\prod_P (1-|P|^{-2}) = \zeta(2)^{-1} = 1-q^{-1}$. 
\item Using equation \eqref{px}, we have $[x^2] P =  \frac{\bc(1/q) (1+q^{-1})}{4} - \frac{\bc'(1/q) (1-q^{-1})}{4q} $, while $[x^2] R = \frac{A(0,0)}{4} + \frac{A_1(0,0)+A_2(0,0)}{8 \log q}$. We compute $\bc'(1/q) = -q \bc(1/q) b_1$, where
$$b_1= \sum_P  \frac{ d(P) ( 3|P|^2-2|P|-1)}{(|P|+1)(|P|^3+2|P|^2-2|P|+1)}.$$ From the definition of $A(1/2;z_1,z_2)$ in \eqref{az}, we also compute that
$$ \frac{A_1(0,0)}{\log q}=\frac{A_2(0,0)}{\log q} = A(0,0) \left(  b_1 + \sum_P \frac{2 d(P)}{|P|^2-1} \right) = A(0,0) \left( b_1 + \frac{2}{q-1} \right),$$ where the last identity comes from the expression of the logarithmic derivative of $\zeta(s)$. Combining all of the above will give the desired identity.
\item We have that $$[x] P =  \frac{11 \bc(1/q) (1-q^{-1})}{24} + \frac{3 \bc'(1/q)(1-q^{-1})}{2q}  - \frac{2 \bc'(1/q)}{q} + \frac{ \bc''(1/q) (1-q^{-1})}{2q^2}  - \frac{\zeta(2)}{2} (\mathcal{F}'(1)+\mathcal{F}''(1)) $$ and
$$[x]R = \frac{11}{24} A(0,0)+ \frac{A_1(0,0)+A_2(0,0)}{\log q}+\frac{A_{12}(0,0)}{2 (\log q)^2}.$$ From the definition of $\mathcal{B}(u)$, we get that
$\bc''(1/q)=q^2 \bc(1/q) (b_1^2+b_1+b_2),$ with $b_1$ as before and
$$b_2= - \sum_P \frac{ d(P)^2 |P| (3-5|P|-2|P|^2-14|P|^3-|P|^4+3|P|^5)}{(|P|+1)^2 (|P|^3+2|P|^2-2|P|+1)^2}.$$ Also $\mathcal{F}(1) \zeta(2) = A(0,0), \mathcal{F}'(1)=0$ and $\mathcal{F}''(1)=\mathcal{F}(1) b_3$, where
$$b_3= - \sum_P \frac{ d(P)^2 |P| (2-4|P|+4|P|^2+2|P|^3)}{(|P|^3+2|P|^2-2|P|+1)^2}.$$ We compute
$$\frac{A_{12}(0,0)}{(\log q)^2}= A(0,0) \left( \left(b_1+\frac{2}{q-1} \right)^2+b_2-b_3- \sum_P  \frac{4 d(P)^2 |P|^2}{(|P|^2-1)^2} \right),$$ and using the fact that
$ \sum_P  \frac{4 d(P)^2 |P|^2}{(|P|^2-1)^2} = \frac{4}{(q-1)^2}+\frac{4}{q-1}$, we have
$$\frac{A_{12}(0,0)}{(\log q)^2}=  A(0,0) \left( \left(b_1+\frac{2}{q-1} \right)^2+b_2-b_3- \frac{4}{(q-1)^2} -\frac{4}{q-1} \right).$$
Doing the computations, we can check that $[x]P=[x]R$.
\item From equation \eqref{px}, we have
\begin{align*} [x^0]P &=  \frac{ \bc(1/q) (1+q^{-1})}{4} - \frac{ \bc' (1/q) (1-q^{-1})}{4q} + \frac{ 2 \bc'(1/q)}{q^2} +\frac{2 \bc''(1/q)}{q^3} - \frac{ \bc^{(3)} (1/q) (1-q^{-1})}{3q^3}  \\
&- \zeta(2) \left[ 2\mathcal{F}'(1)+4\mathcal{F}''(1)+\mathcal{F}^{(3)}(1)+ \frac{\alpha'(1)(\mathcal{F}(1)+\mathcal{F}'(1))}{2}+ \frac{\alpha(1)(\mathcal{F}'(1)+\mathcal{F}''(1))}{2}+\frac{\mathcal{F}(1) \alpha''(1)}{2} \right]. \end{align*} Also from \eqref{conjpol},
\begin{align*}
[x^0]R &= \frac{A(0,0)}{4} + \frac{1}{24 \log q} (A_1(0,0)+A_2(0,0)) + \frac{1}{(\log q)^2} A_{12}(0,0) \\
&- \frac{1}{12 (\log q)^3} (A_{222}(0,0)-3A_{122}(0,0)-3A_{112}(0,0)+A_{111}(0,0)). \end{align*}
We compute
$$\bc^{(3)}(1/q) = -q^3\bc(1/q) (b_1^3+3 b_1 b_2 + b_4 + 3b_1^3+3 b_2+2 b_1),$$ where
$$b_4 = \sum_P \frac{ d(P)^3 |P| (-3+6|P|-3|P|^2+91|P|^3-41|P|^4-29|P|^5-57|P|^6-55|P|^7-8|P|^8+3|P|^9)}{(|P|+1)^3 (|P|^3+2|P|^2-2|P|+1)^3}.$$
Also
$\mathcal{F}^{(3)}(1) = -3\mathcal{F}(1)b_3, \alpha(1) = 2 \displaystyle \left( b_1+ \frac{2}{q-1} \right),$
$$ \alpha'(1) = b_5 = \sum_P \frac{4 d(P)^2 |P|}{|P|^3+2|P|^2-2|P|+1},$$ and
$$\alpha''(1) = -b_5+b_6,$$ where
$$b_6 = \sum_P \frac{4 d(P)^3 |P|(|P|^3+|P|^2-1)}{(|P|^3+2|P|^2-2|P|+1)^2}.$$ 
Using \eqref{az}, 
$$\frac{A_{111}(0,0)}{ (\log q)^3}=\frac{A_{222}(0,0)}{(\log q)^3} = A(0,0) \left( \left(b_1+\frac{2}{q-1} \right)^3+3 \left( b_1+\frac{2}{q-1} \right)b_7+b_8 \right),$$ where
$$b_7 = -\sum_P \frac{ d(P)^2 |P| (5-21|P|+32|P|^2-16|P|^3-5|P|^4+9|P|^5)}{(|P|-1)^2 (|P|^3+2|P|^2-2|P|+1)^2},$$ and
$$b_8 = \sum_P \frac{ d(P)^3 |P|(9-46|P|+81|P|^2-35|P|^3-43|P|^4+29|P|^5+35|P|^6-29|P|^7-2|P|^8+17|P|^9)}{(|P|^4+|P|^3-4|P|^2+3|P|-1)^3}.$$
Similarly we compute
\begin{align*}
\frac{A_{122}(0,0)}{(\log q)^3} &= \frac{A_{112}(0,0)}{(\log q)^3} = A(0,0) \bigg( \left(b_1+\frac{2}{q-1} \right)^3+2 \left(b_1+\frac{2}{q-1} \right) \left( b_2-b_3-\frac{4}{(q-1)^2}-\frac{4}{q-1} \right) \\
&+ \left(b_1+\frac{2}{q-1} \right) b_7  +b_9 \bigg),\end{align*} where
\begin{align*}
\frac{b_9}{2}& = \frac{b_8}{6}-\frac{b_6}{2}+\frac{b_4}{3} + \sum_P \frac{ 8 d(P)^3 |P|^2(|P|^2+1)}{3(|P|^2-1)^3} \\
&= \frac{b_8}{6}-\frac{b_6}{2}+\frac{b_4}{3} + + \frac{8q(q+1)}{3(q-1)^3}. \end{align*} Combining all of the above will give the desired identity of coefficients.
\end{enumerate}
\end{remark}
\label{last}

\section{Proof of Theorem \ref{th2}}
\label{thirdmom}
Here we will prove Theorem \ref{th2}. Computing the third moment is similar to the computation of the second moment, so we will skip some of the details. 

Recall from section \ref{setup} that $S_{3g}= M_{3g}+S_{3g}(V = \square) + S_{3g}(V \neq \square)+S_{3g,\text{o}}+O(q^{3g/2(1+\epsilon)}),$ and a similar expression holds for $M_{3g-1}$.
\subsection{Main term}
Here we focus on the main term $M_{3g}$. Recall that 
\begin{equation}    M_{3g} = q^{2g+1}  \Big( 1-\frac{1}{q} \Big) \sum_{\substack{f \in \mathcal{M}_{\leq 3g} \\ f = \square}} \frac{d_3(f)}{|f|^{\frac{3}{2}}}  \phi(f)\sum_{\substack{C \in \mathcal{M}_{\leq g-1} \\ C | f^{\infty}}} \frac{1}{|C|^2}. \label{principal2} \end{equation}
We have the following.
\begin{lemma}
With the same notation as before,
$$M_{3g}+M_{3g-1} = \frac{q^{2g+1}}{\zeta(2)} Q_1(2g+1)+O(q^{g(1+\epsilon)}),$$ where $Q_1$ is a polynomial of degree $6$.
\end{lemma} \label{thirdmain}
\begin{proof}
Similarly as in section \ref{main}, we rewrite
\begin{equation}
M_{3g} =  \frac{q^{2g+1}}{\zeta(2)}\sum_{l \in \mathcal{M}_{\leq \left[\frac{3g}{2} \right]}}  \frac{d_3(l^2)}{|l| \prod_{P|l} \left(1+\frac{1}{|P|} \right)} + O(q^{g \epsilon}).  \label{m3} \end{equation}
Let $$\mathcal{A}_3(u) = \sum_{l \in \mathcal{M}} u^{d(l)} \frac{d_3(l^2)}{ \prod_{P|l} \left(1+\frac{1}{|P|} \right)}.$$ Using Euler products, we get that
\begin{align*} \mathcal{A}_3(u) &= \prod_P \left( 1+  \frac{u^{d(P)}(6-3u^{d(P)}+u^{2d(P)})}{(1+ \frac{1}{|P|})(1-u^{d(P)})^3} \right) \\
&= \mathcal{Z}(u)^6 \mathcal{B}_3(u),\end{align*} where
\begin{equation}
\mathcal{B}_3(u)= \prod_{P} \left(1 - \frac{6u^{d(P)}-(15-6|P|)u^{2d(P)}+(20 - 8|P|) u^{3d(P)}-(15-3|P|)u^{4d(P)}+6u^{5d(P)}-u^{6d(P)}}{|P|+1} \right) .\label{bu}
\end{equation} From the expression of $\mathcal{B}_3(u)$ above, note that it converges absolutely for $|u|<\frac{1}{\sqrt{q}}$. We can further write
$$ \mathcal{B}_3(u) = \frac{ \mathcal{Z}(u^4)^6}{\mathcal{Z}(u^2)^6} \mathcal{C}(u),$$ where $\mathcal{C}(u)$ converges absolutely for $|u|< \frac{1}{q^{1/3}}.$ From the above we see that $\mathcal{B}_3(u)$ has an analytic continuation for $|u|<\frac{1}{q^{1/3}}$. 

Now using \eqref{perron} in \eqref{m3}, we get that
\begin{equation}
M_{3g} = \frac{q^{2g+1}}{\zeta(2)} \frac{1}{2\pi i} \oint_{|u|=r_1} \frac{\mathcal{B}_3(u)}{(1-qu)^7 (qu)^{[3g/2]}} \, \frac{du}{u}+O(q^{g \epsilon}), \label{int1}
\end{equation} where $r_1<1/q$.
Similarly
\begin{equation}
M_{3g-1}= \frac{q^{2g+1}}{\zeta(2)} \frac{1}{2\pi i} \oint_{|u|=r_1} \frac{\mathcal{B}_3(u)}{(1-qu)^7 (qu)^{[(3g-1)/2]}} \, \frac{du}{u}+O(q^{g \epsilon}). \label{int2}
\end{equation}
Note that in the two integrals above, by shifting the contour of integration to a circle around the origin of radius $R= q^{-1/3-\epsilon}$, we encounter a pole at $u=1/q$. Since $\mathcal{B}_3(u)$ has an analytic continuation for $|u|<q^{-1/3}$, we see that 
$$ M_{3g}=- \text{Res}(u=1/q) + O(q^{g(1+\epsilon)}),$$ and a similar formula holds for $M_{3g-1}$. By computing the residues at $u=1/q$ for $M_{3g}$ and $M_{3g-1}$, Lemma \ref{thirdmain} follows.
\end{proof}

\subsection{Secondary main term}
Here we will evaluate $S_3(V=\square)= S_{3g}(V=\square)+S_{3g-1}(V=\square)$, with $S_{3g}(V=\square)$ given by \eqref{s1s}. We'll prove the following.
\begin{lemma}
 With the same notation as before, we have
 $$ S_3(V=\square) = \frac{q^{2g+1}}{\zeta(2)} Q_2(2g+1)+O(q^{3g/2(1+\epsilon)}), $$ where $Q_2$ is a polynomial of degree $6$.
\label{secondarythird}
\end{lemma}
Before proving the above, we will first state two additional lemmas. We'll omit the proofs.
\begin{lemma}
 Let $V$ be a monic polynomial in $\mathbb{F}_q[x]$. For $|z|>1/q^2$, let
$$\mathcal{A}(V;z,w) = \sum_{f \in \mathcal{M}} w^{d(f)} \frac{d_3(f) G(V,\chi_f)}{ \sqrt{|f|} \prod_{P|f} \left(1- \frac{1}{|P|^2z^{d(P)}} \right)}.$$ Then we have

(a) $$\mathcal{A}(V;z,w) = \mathcal{L}(w,\chi_V)^3 \prod_P \mathcal{H}_P(V;z,w), $$ where
$$\mathcal{H}_P(V;z,w) = 
\begin{cases}
1+ \frac{3 \left( \frac{V}{P} \right) w^{d(P)}}{|P|^2z^{d(P)}-1}+3w^{2d(P)}- \frac{9w^{2d(P)}}{1-\frac{1}{|P|^2z^{d(P)}}}- \left( \frac{V}{P} \right)w^{3d(P)}+ \frac{9 \left( \frac{V}{P} \right) w^{3d(P)}}{1-\frac{1}{|P|^2z^{d(P)}}}- \frac{3w^{4d(P)}}{1-\frac{1}{|P|^2z^{d(P)}}} & \mbox{if } P \nmid V \\
1+ \left(1-\frac{1}{|P|^2 z^{d(P)}}\right)^{-1} \sum_{i=1}^{\infty} \frac{w^{id(P)}d_3(P^i) G(V,\chi_{P^i})}{|P|^{i/2}} & \mbox{if } P|V
\end{cases}
$$
(b) If $V=l^2$ and $l \in \mathcal{M}$, then 
$$ \mathcal{A}(l^2;z,w) =  \mathcal{Z}(w)^3 \prod_P \mathcal{A}_P(l^2;z;w),$$ where
$$\mathcal{A}_P(l^2;z,w) =
\begin{cases}
 1+ \frac{3w^{d(P)}}{|P|^2z^{d(P)}-1}+3w^{2d(P)}- \frac{9w^{2d(P)}}{1-\frac{1}{|P|^2z^{d(P)}}}   -w^{3d(P)}+\frac{9w^{3d(P)}}{1-\frac{1}{|P|^2z^{d(P)}}}- \frac{3 w^{4d(P)}}{1-\frac{1}{|P|^2z^{d(P)}}} & \mbox{ if } P \nmid l \\
 (1-w^{d(P)})^3 \left( 1+ \sum_{i=1}^{\infty} \frac{w^{id(P)}d_3(P^i) G(l^2,\chi_{P^i})}{|P|^{i/2} \left(1- \frac{1}{|P|^2z^{d(P)}} \right)} \right) & \mbox{ if } P|l \end{cases}
$$
\label{sumf}
\end{lemma}
 We also have the following.
\begin{lemma}
Keeping the notation from the previous lemma, let $$C(z,w) = \sum_{l \in \mathcal{M}} z^{d(l)} \prod_P \mathcal{A}_P(l^2;z,w).$$ 

(a) Then $$ \mathcal{C}(z,w) =  \mathcal{Z}(z) \mathcal{Z}(qw^2z)^6 \mathcal{Z} \left( \frac{1}{q^2z} \right) \frac{\mathcal{Z}(w^2z^2)^3}{\mathcal{Z}(wz)^3} \mathcal{H}(z,w),$$ where $\mathcal{H}(z,w)= \prod_P \mathcal{H}_P(z,w)$, and 
\begin{align*}
&\mathcal{H}_P(z,w) =(1-w^{d})^3 (1-|P|(w^2z)^{d})^3(1+(wz)^{d})^3  \Big(1+ 3 w^{d}+3|P|(zw^2)^{d}+ \frac{3w^{2d}}{|P|} - 3 (zw)^{d} - \frac{1}{|P|^2z^{d}} \\
& -6(zw^2)^{d} +|P|(zw^3)^{d} -3(zw^4)^{d}-|P|(z^2w^3)^{d} +3|P|(z^2w^4)^{d}+|P|(z^2w^6)^{d}-|P|^2(z^3w^6)^{d} \Big) .  \end{align*} (Here, $d$ stands for $d(P)$.)
Moreover, $\mathcal{H}(z,w)$ converges absolutely for $|w|<q^{-1/2}, |zw|<q^{-1/2}, |zw^2|<q^{-3/2},|z|>q^{-1}$.

(b)  We have
\begin{align*}
 \mathcal{H}_P \left(z, \frac{1}{q} \right) &= \left(1 -\frac{1}{|P|} \right)^3 \Big(1+\frac{3}{|P|}+\frac{3}{|P|^3} - \frac{1}{|P|^2z^{d(P)}} -\frac{5z^{d(P)}}{|P|^2}  -\frac{4z^{2d(P)}}{|P|^2} -\frac{6z^{2d(P)}}{|P|^3} - \frac{8z^{2d(P)}}{|P|^5}+\frac{14z^{3d(P)}}{|P|^4} \\
 &+ \frac{6z^{3d(P)}}{|P|^6} + \frac{6z^{4d(P)}}{|P|^4}  +\frac{6z^{4d(P)}}{|P|^7} -\frac{12 z^{5d(P)}}{|P|^6}  -\frac{8z^{5d(P)}}{|P|^8} - \frac{4z^{6d(P)}}{|P|^6} +\frac{6z^{6d(P)}}{|P|^7} +\frac{2z^{7d(P)}}{|P|^8} +\frac{3z^{7d(P)}}{|P|^{10}}\\
 & +\frac{z^{8d(P)}}{|P|^{8}} -\frac{3z^{8d(P)}}{|P|^{9}} -\frac{z^{8d(P)}}{|P|^{11}}  +\frac{z^{9d(P)}}{|P|^{10}}   \Big),
 \end{align*} and $ \mathcal{H}\left(z,\frac{1}{q} \right)$ converges absolutely for $q^{-1}<|z|<\sqrt{q}$ and has an analytic continuation when $q^{-1}<|z|<q$.
\label{sumv}
\end{lemma}
\begin{proof}[Proof of Lemma \ref{secondarythird}]
Recall that 
\begin{align*}
S_{3g}(V= \square) &= q^{2g+1}\sum_{\substack{f \in \mathcal{M}_{\leq 3g} \\ d(f) \text{ even}}} \frac{d_3(f)}{|f|^{\frac{3}{2}}}\sum_{\substack{C \in \mathcal{M}_{\leq g-1} \\ C | f^{\infty}}} |C|^{-2} \bigg[ (q-1)  \sum_{ l \in \mathcal{M}_{\leq \frac{d(f)}{2}-g-2+d(C)} } G(l^2, \chi_f) - \sum_{l \in \mathcal{M}_{\frac{d(f)}{2}-g-1+d(C)}} G(l^2,\chi_f) \nonumber \\
  &- \frac{q-1}{q}  \sum_{ l \in \mathcal{M}_{\leq \frac{d(f)}{2}-g-1+d(C)} } G(l^2, \chi_f) + \frac{1}{q} \sum_{l \in \mathcal{M}_{\frac{d(f)}{2}-g+d(C)}} G(l^2,\chi_f) \bigg] . 
  \end{align*} 
  We proceed similarly as in section \ref{secondary}, and after using Lemmas \ref{sumf} and \ref{sumv}, we get that
  $$S_{3g}(V=\square) = -q^{2g+1}  \frac{1}{(2 \pi i)^2} \oint_{|z|=r_1} \oint_{|w| =r_2}  \frac{z^g(q^2w^2z)^{-[3g/2]} (1-qwz)^3}{w(1-z)(1-qw)^3(1-q^2w^2z)^7(1-qw^2z^2)^3} \mathcal{H}(z,w) \, dw \, dz+O(q^{3g/2(1+\epsilon)}),$$ where $r_2<1/q$ and $r_1=q^{\epsilon-1}$. We enlarge the contour $|w|=r_2$ to $|w|=q^{-1/2-\epsilon}$, and we encounter a pole of order $3$ at $w=1/q$. By Lemma \ref{sumv}, $\mathcal{H}(z,w)$ is analytic in this region. When $r_1=q^{\epsilon-1}$ and $|w|=q^{-1/2-\epsilon}$, the double integral is bounded by $O(q^{-g(1-\epsilon)})$.

 We evaluate the residue at $w=1/q$, and we get that
 \begin{equation*}
 S_{3g}(V=\square) =  -q^{2g+1} \frac{1}{2 \pi i} \oint_{|z|=q^{\epsilon-1}} \frac{\mathcal{H}(z,1/q)}{z^{[3g/2]-g} (z-1)^7} P_1(z,g) \, dz ,
 \end{equation*} where $P_1(z,g)$ is a polynomial in $z$ and $g$.
  In the expression for $S_{3g}(V=\square)$ above, the integrand has a pole of order $7$ at $z=1$. By Lemma \ref{sumv}, $\mathcal{H}(z,1/q)$ is analytic for $q^{-1}<|z|<\sqrt{q}$ and has an analytic continuation when $|z|<q$. By shifting the contour of integration to $|z|=q^{1-\epsilon}$, we encounter the pole at $z=1$ and we bound the integral over the new contour by $O(q^{3g/2(1+\epsilon)})$. We do the same for $S_{3g-1}(V=\square)$, and adding the two terms gives that $S_{3}(V=\square) = \frac{q^{2g+1}}{\zeta(2)} Q_2(2g+1)+O(q^{3g/2(1+\epsilon)}),$ where the polynomial $Q_2$ has degree $6$ and can be computed explicitly by evaluating the residue at $z=1$. This finishes the proof of Lemma \ref{secondarythird}.
 \end{proof}
  \subsection{Error from non-square $V$}
Here we'll show that $S_{3g}(V \neq \square) \ll q^{3g/2(1+\epsilon)}$. The proof is similar to the one in section \ref{error}. It is enough to bound the term $S_{3g,\text{o}}$ given by \eqref{odd} (bounding $S_{3g,\text{e}}(V \neq \square)$ follows in the same way.) In equation \eqref{odd}, we write $S_{3g,\text{o}}$ as a difference of two terms. Similarly as in section \ref{error}, we want to bound
\begin{equation}
S_{3,\text{o}} = q^{2g+1} \sqrt{q} \frac{1}{2 \pi i} \oint_{|u|=r_1} \sum_{\substack{n=0 \\ n \text{ odd}}}^{3g} q^{-n}   \sum_{i=0}^g \frac{1}{q^{2i} u^{i+1}} \sum_{V \in \mathcal{M}_{n-2g-2+2i}} \sum_{f \in \mathcal{M}_n} \frac{ d_3(f) G(V,\chi_f)|f|^{-1/2}}{\displaystyle \prod_{P|f} (1-u^{d(P)})} \, du,  \label{err4}
\end{equation} where $r_1<1$. By Lemma \ref{sumf}, 
$$  \sum_{f \in \mathcal{M}_n}  \frac{d_3(f) G(V,\chi_f) |f|^{-1/2}}{ \displaystyle \prod_{P|f} (1- u^{d(P)})} = \frac{1}{2 \pi i} \oint_{|w|=r_2} \frac{\mathcal{L}(w,\chi_V)^3 \prod_P \mathcal{H}_P(V;1/(q^2u),w)}{w^{n+1}} \, dw,$$ where $\prod_P  \mathcal{H}_P(V;1/(q^2u),w)$ converges for $|wu|<1/q, |w|<1/\sqrt{q}$ and $|u|<1$. We pick $r_1=q^{-\epsilon}$ and $r_2=q^{-1/2-\epsilon},$ and let $k$ be minimal such that $r_1^kr_2<1/q$. Then
 $$ \prod_P \mathcal{H}_P(V;1/(q^2u),w) = \mathcal{L}(wu,\chi_V)^3 \mathcal{L}(wu^2,\chi_V)^3 \cdot \ldots \cdot \mathcal{L}(wu^{k-1}, \chi_V)^3 \mathcal{D}(V;w,u),$$ where $\mathcal{D}(V;w,u)$ is given by a converging Euler product. Then 
 \begin{equation*}
 \left|  \sum_{f \in \mathcal{M}_n}  \frac{d_3(f) G(V,\chi_f) |f|^{-1/2}}{ \displaystyle \prod_{P|f} (1- u^{d(P)})} \right| \ll q^{n/2(1+\epsilon)} \left| \mathcal{L}(w,\chi_V) \cdot \ldots \cdot \mathcal{L}(wu^{k-1},\chi_V) \right|^3.  \end{equation*} Combining this with the bound
 $$ | \mathcal{L}(wu^j,\chi_V) | \ll e^{ \frac{n-2g+2i}{ 2 \log_q(n/2-g+i)} + 4\sqrt{q(n-2g+2i)}}, $$ for $j \in \{0, \ldots, k-1 \}$, and trivially bounding the sum over $V$ gives that $ S_{3,\text{o}} \ll q^{3g/2(1+\epsilon)}$. Then $S_{3g}(V \neq \square) \ll q^{3g/2(1+\epsilon)}.$ 
 Combining this bound with Lemmas \ref{secondarythird} and \ref{thirdmain} and putting $Q(x)=Q_1(x)+Q_2(x)$ finishes the proof of Theorem \ref{th2}. 
 
 \begin{remark}
 We note that 
$$[x^6]Q_1= \frac{ 729}{2^{11} 6! } \mathcal{B}_3(1/q),$$ which follows from evaluating the residues in the integrals \eqref{int1} and \eqref{int2}. Also
$$[x^6]Q_2= - \frac{217}{2^{11} 6!} \mathcal{H}(1,1/q) \zeta(2)^4,$$ which follows from evaluating the residue at $z=1$ in the integral for $S_{3g}(V=\square)$ above. By direct computation, we have that $\mathcal{B}_3(1/q)= \mathcal{H}(1,1/q) \zeta(2)^4 = A_3 \left(\frac{1}{2};0,0,0 \right),$ where $A_3 \left(\frac{1}{2};0,0,0 \right)$ is given by equation \eqref{a3}. Combining the two equations above and since $Q=Q_1+Q_2$, we have
$$[x^6] Q = \frac{1}{2880} A_3 \left(\frac{1}{2};0,0,0 \right),$$ which matches the leading coefficient in the conjectured formula \eqref{thirdconj}.
 \end{remark} 
 \break
 \textsc{Acknowledgments.} I would like to thank Kannan Soundararajan for his suggestions and for the many helpful discussions we've had while working on this problem.
\bibliography{biblvariance}
\bibliographystyle{plain}

\end{document}